\documentclass[graybox]{svmult}

\usepackage{mathptmx}       
\usepackage{helvet}         
\usepackage{courier}        
\usepackage{type1cm}        
\usepackage{makeidx}         
\usepackage{graphicx}        
\usepackage{multicol}        
\usepackage[bottom]{footmisc}

\makeindex             
\usepackage{amsmath}
\usepackage{amsfonts}
\usepackage{amssymb}
\usepackage{mathrsfs}
\usepackage{tikz}
\usepackage{multirow}
\DeclareMathAlphabet{\mathcal}{OMS}{cmsy}{m}{n}
\usepackage{bm}

\usepackage[unicode=true,bookmarks=false,breaklinks=false,
pdfborder={0 0 1},backref=section,colorlinks=false]{hyperref}

\begin{document}

\title*{Metastable Hierarchy in Abstract Low-Temperature Lattice Models}
\author{Seonwoo Kim}
\institute{Seonwoo Kim \at June E Huh Center for Mathematical Challenges, Korea Institute for Advanced Study, 85 Hoegi-ro, Dongdaemun-gu, Seoul 02455, Republic of Korea, \email{seonwookim@kias.re.kr}}
%
%
\maketitle


\abstract{In this article, we review the metastable hierarchy in low-temperature lattice models. In the first part, we state that for any abstract lattice system governed by a Hamiltonian potential and evolving according to a Metropolis-type dynamics, there exists a hierarchical decomposition of the collection of stable plateaux in the system into multiple $\mathfrak{m}$ levels, such that at each level there exist tunneling metastable transitions between the stable plateaux, which can be characterized by convergence to a simple Markov chain as the inverse temperature $\beta$ tends to infinity. In the second part, we collect several examples that realize this hierarchical structure of metastability. In order to fix the ideas, we select the Ising model as our lattice system and discuss its metastable behavior under four different types of dynamics, namely the Glauber dynamics with positive/zero external fields and the Kawasaki dynamics with few/many particles. This review article is submitted to the proceedings of the event \emph{PSPDE XII}, held at the University of Trieste from September 9--13, 2024.}


\section{\label{sec1}Introduction}

Metastability occurs in a wide class of low-temperature dynamical
systems that possess two or more locally stable states, including
countless real-world examples such as the supercooling of water or
the persistence of stock prices at high levels. In the context of
statistical physics, this phenomenon can be understood as a first-order
phase transition with respect to intrinsic physical parameters such
as magnetization or particle density. It occurs in a large number
of concrete models falling in the range of small random perturbations
of dynamical systems \cite{BEGK04,FW,LanLeeSeo Meta-1,LanLeeSeo Meta-2,LanMarianiSeo},
condensing interacting particle systems \cite{BL ZRP,BDG,Kim21a,Kim23,LMS critical,Seo NRZRP},
ferromagnetic spin systems at low temperature \cite{BAC,BdHN,KS22,KS24,MNOS04,NZ19,NS91,NS92,OS95},
etc.

During the past several decades, there were several paradigmatic discoveries
to analyze and quantify the phenomenon of metastability with mathematically rigorous methods.
Below we gather a brief, non-exhaustive list of the known approaches.

\begin{itemize}
\item {\it Pathwise approach}: Freidlin and Wentzell \cite{FW} developed a theory of
typical trajectories of a given dynamical system by means of the well-known large deviation theory.
Inspired by this discovery, Cassandro, Galves, Olivieri and Vares \cite{CGOV} adopted this theory to discrete
ferromagnetic systems with applications to the Curie--Weiss model
and the contact process on $\mathbb{Z}$. See \cite{OV} for an extensive history, overview and literature on the subject.
\item {\it Potential theory approach}: Bovier, Eckhoff, Gayrard and Klein \cite{BEGK04,BGK05}
made a remarkable breakthrough by interpreting the well-known \emph{potential
theory} in the language of metastability. We refer the readers to a recent monograph \cite{BdH} for detailed
aspects and complete literature on the potential-theoretic approach to metastability.
\item {\it Martingale approach}: In a series of papers \cite{BL10,BL12b,BL15a}, Beltrán and Landim developed the martingale
approach to metastability, that is, they characterized the successive metastable jumps among the valleys
as an asymptotically Markovian simple jumps between the (meta)stable states, by means of the martingale
relation of Markov processes. See \cite{Lan19} for a detailed explanation.
\item {\it Resolvent approach}: Landim, Marcondes and Seo \cite{LMS resolvent} introduced
a completely new methodology, namely the \emph{resolvent approach},
to study metastability occurring in systems with a complicated structure
of multiple stable states. The key idea is that the metastable behavior of a given dynamics is
\emph{equivalent} to the flatness of solutions to particular resolvent equations.
\end{itemize}

In this article, we consider a general class
of low-temperature ferromagnetic lattice models and identify their hierarchical
structure of tunneling transitions that occur in the complicated energy
landscape of metastable states. To fix ideas, we consider the Metropolis-type
dynamics (cf. \eqref{eq:rbeta def}) in the state space such that
the dynamics jumps with rate $1$ to equal or lower energy configurations
and with $\beta$-exponentially small rate to higher energy configurations,
where $\beta$ denotes the inverse temperature. We characterize each \emph{stable plateau} (see Definition
\ref{def:stab plat}) as a metastable element in the system and quantify
the metastable transitions at each level with a limiting Markov chain
between these stable plateaux, by accelerating the dynamics and proving
convergence to a limit object. The exact formulation of this convergence
is given in Section \ref{sec2.2}.

More specifically, denote temporarily by $\mathscr{P}^{1}$ the collection
of stable plateaux in the system. For each stable plateau, its initial
depth is defined as the minimum energy barrier to make a transition
to another one. Then, the initial metastable tunneling transitions
occur between those with the minimum initial depth. In turn, stable
plateaux in each irreducible class subject to the first-level transitions
form a new element in the second level, so that next we can characterize
the second-level tunneling transitions in the new collection $\mathscr{P}^{2}$
in a similar way. We can iterate this procedure until we obtain a
unique irreducible class of transitions at the terminal level, say
level $\mathfrak{m}$. Then, the sequence $\mathscr{P}^{1},\mathscr{P}^{2},\dots,\mathscr{P}^{\mathfrak{m}}$
contains complete information about the metastable hierarchy in the
system. See Section \ref{sec2.1} for the exact formulation of this
scheme.

We chose the Metropolis-type dynamics (cf. \eqref{eq:rbeta def}) as our Markov chain in the abstract
system to avoid further technical difficulties irrelevant to the main
results of the article. Nevertheless, we believe that the main ideas
presented here are robust enough to be directly applicable to a broader
class of stochastic dynamics, such as non-reversible dynamics in the
category of Freidlin--Wentzell-type Markov chains (cf. \cite[Section 6]{OV}
or \cite{CNSoh}).

To demonstrate this method, we fix a concrete system, namely the Ising model, given
four types of dynamical systems defined thereon, which are:
\begin{enumerate}
\item Glauber dynamics on the Ising model with positive external field;
\item Glauber dynamics on the Ising model with zero external field;
\item Kawasaki dynamics on the Ising model with few particles;
\item Kawasaki dynamics on the Ising model with many particles.
\end{enumerate}

The results are based on the works from \cite{NS91,BL11} (first), \cite{NZ19,KS25} (second), \cite{BL15b} (third), and \cite{Kim24} (fourth). In this review article, we omit most of the proofs of the main results; interested readers are referred to the full version \cite{Kim24} for the hierarchical decomposition and the above-mentioned references for the concrete examples.

In Section \ref{sec2}, we study the general framework for establishing the hierarchical
structure of metastable transitions in low-temperature lattice models.
In Section \ref{sec3}, we apply this framework to the four types of models listed above.
In Section \ref{sec4}, we discuss some possible extensions.

\section{\label{sec2}Metastable Hierarchy in Low-Temperature Lattice Models}

Consider a finite state space $\Omega$ which is given an (undirected) edge
structure, such that we write $\eta\sim\xi$ (and $\xi\sim\eta$)
if and only if $\{\eta,\xi\}$ is an edge. We assume that the resulting
graph $\Omega$ is connected.

Suppose that a \emph{Hamiltonian} (or \emph{energy}) function $\mathbb{H}:\Omega\to\mathbb{R}$
is given to the system, such that the corresponding \emph{Gibbs measure}
$\mu_{\beta}$ on $\Omega$ is defined as
\begin{equation}
\mu_{\beta}(\eta):=Z_{\beta}^{-1}e^{-\beta\mathbb{H}(\eta)}\quad\text{for}\quad\eta\in\Omega,\label{eq:Gibbs def}
\end{equation}
where $\beta>0$ is the inverse temperature of the system and $Z_{\beta}:=\sum_{\eta\in\Omega}e^{-\beta\mathbb{H}(\eta)}$
is the normalizing constant such that $\mu_{\beta}$ is a probability
measure. We are interested in the regime of zero-temperature limit:
$\beta\to\infty$.
\begin{definition}
\label{nota:bot bdry def}For $\mathcal{A}\subseteq\Omega$, define
the \emph{bottom} of $\mathcal{A}$ as $\mathcal{F}(\mathcal{A}):= {\rm argmin}_\mathcal{A} \mathbb{H}$
and the (outer) \emph{boundary} of $\mathcal{A}$ as $\partial\mathcal{A}:=\{\eta\in\Omega\setminus\mathcal{A}:\eta\sim\xi, \ \exists\xi\in\mathcal{A}\}$.
Moreover, define $\partial^{\star}\mathcal{A}:=\mathcal{F}(\partial\mathcal{A})$.
\end{definition}

Consider a \emph{Metropolis-type} dynamics $\{\eta_{\beta}(t)\}_{t\ge0}$
in $\Omega$ whose transition rate function $r_{\beta}:\Omega\times\Omega\to[0,\infty)$
is represented as
\begin{equation}
r_{\beta}(\eta,\xi)=\begin{cases}
e^{-\beta\max\{\mathbb{H}(\xi)-\mathbb{H}(\eta),0\}} & \text{if}\quad\eta\sim\xi,\\
0 & \text{otherwise}.
\end{cases}\label{eq:rbeta def}
\end{equation}
Denote by $L_{\beta}$ the corresponding infinitesimal generator.
Since $\Omega$ is connected, $\{\eta_{\beta}(t)\}_{t\ge0}$ is irreducible
in $\Omega$. One can easily notice that the dynamics is
reversible with respect to $\mu_{\beta}$:
\begin{equation}
\mu_{\beta}(\eta)r_{\beta}(\eta,\xi)=\mu_{\beta}(\xi)r_{\beta}(\xi,\eta)=\begin{cases}
Z_{\beta}^{-1}e^{-\beta\max\{\mathbb{H}(\eta),\mathbb{H}(\xi)\}} & \text{if}\quad\eta\sim\xi,\\
0 & \text{otherwise}.
\end{cases}\label{eq:det bal}
\end{equation}
We denote by $\mathcal{S}:=\mathcal{F}(\Omega)$ the minimizer of
the Hamiltonian (cf. Definition \ref{nota:bot bdry def}). The elements
of $\mathcal{S}$ are called \emph{ground states}. By \eqref{eq:Gibbs def},
it is straightforward that $\lim_{\beta\to\infty}\mu_{\beta}(\mathcal{S})=1$.

A sequence $\omega=(\omega_{n})_{n=0}^{N}$ of configurations is called
a \emph{path} from $\omega_{0}=\eta$ to $\omega_{N}=\xi$ if $\omega_{n}\sim\omega_{n+1}$
for all $n\in[0,N-1]$.\footnote{In this article, $[\alpha,\alpha']$ denotes the collection of integers
from $\alpha$ to $\alpha'$.} In this case, we write $\omega:\eta\to\xi$. Moreover, writing
$\omega:\mathcal{A}\to\mathcal{B}$ for disjoint subsets $\mathcal{A},\mathcal{B}$
implies that $\omega:\eta\to\xi$ for some $\eta\in\mathcal{A}$
and $\xi\in\mathcal{B}$. For each path $\omega$, define the \emph{height}
of $\omega$ as
\begin{equation}
\Phi_{\omega}:=\max_{n\in[0,N]}\mathbb{H}(\omega_{n}).\label{eq:Phi-omega def}
\end{equation}
Then, for $\eta,\xi\in\Omega$, define the \emph{communication height}
or \emph{energy barrier} between $\eta$ and $\xi$ as
\begin{equation}
\Phi(\eta,\xi):=\min_{\omega:\eta\to\xi}\Phi_{\omega}.\label{eq:comm height def}
\end{equation}
For two disjoint subsets $\mathcal{A},\mathcal{B}\subseteq\Omega$,
define $\Phi(\mathcal{A},\mathcal{B}):=\min_{\omega:\mathcal{A}\to\mathcal{B}}\Phi_{\omega}$.
By concatenating the paths, it holds that
\begin{equation}
\Phi(\eta,\xi)\le\max\{\Phi(\eta,\zeta),\Phi(\zeta,\xi)\}\quad\text{for all}\quad\eta,\xi,\zeta\in\Omega.\label{eq:Phi max ineq}
\end{equation}
Now, define
\begin{equation}
\overline{\Phi}:=\max_{\bm{s},\bm{s'}\in\mathcal{S}}\Phi(\bm{s},\bm{s'}).\label{eq:Phi-bar def}
\end{equation}
In other words, $\overline{\Phi}$ is the minimal energy level subject
to which we observe all transitions between the ground states in $\mathcal{S}$.
Thus, we naturally define
\begin{equation}
\overline{\Omega}:=\{\eta\in\Omega:\Phi(\mathcal{S},\eta)\le\overline{\Phi}\},\label{eq:Omega-bar def}
\end{equation}
which is the collection of all configurations reachable from the ground
states by paths of height at most $\overline{\Phi}$. Hereafter, our
analysis is focused on the collection $\overline{\Omega}$, which
is a connected subgraph of $\Omega$.
\begin{remark}
\label{rem:Omega Omega-bar}In the zero-temperature limit, the energy
landscape near the ground states captures all the essential features
of the metastability phenomenon. Thus, instead of trying to characterize
all the locally stable and saddle structures in the whole system,
we focus on the essential subset $\overline{\Omega}$. Nevertheless,
the following analyses below are fully valid even
if $\overline{\Omega}$ is replaced by the full space $\Omega$; which is indeed
the first case handled in Section \ref{sec3}, the Glauber dynamics on the Ising model with positive external field.
\end{remark}

Now, we define two crucial notions explored in this article; namely \emph{stable plateaux} and \emph{cycles}. We call that a set $\mathcal{A}$ is \emph{connected} if for any $\eta,\xi\in\mathcal{A}$, there exists a path $\eta=\eta_0,\eta_1,\dots,\eta_N=\xi$ inside $\mathcal{A}$, i.e., $\eta_n\in\mathcal{A}$ for all $n\in[0,N]$.
\begin{definition}[Stable plateau]
\label{def:stab plat}A nonempty connected set $\mathcal{P}$ is
called a \emph{stable plateau} if the following two statements hold.
\begin{itemize}
\item For all $\eta,\xi\in\mathcal{P}$, $\mathbb{H}(\eta)=\mathbb{H}(\xi)$;
we denote by $\mathbb{H}(\mathcal{P})$ the common energy value.\footnote{We adopt this notation for any set with common energy.}
\item It holds that $\mathbb{H}(\zeta)>\mathbb{H}(\mathcal{P})$ for all
$\zeta\in\partial\mathcal{P}$.
\end{itemize}
\end{definition}

We denote by $\nu_{0}$ the number of stable plateaux in $\overline{\Omega}$
and by $\mathscr{P}^{1}$ their collection:
\begin{equation}
\mathscr{P}^{1}:=\{\mathcal{P}_{i}^{1}:i\in[1,\nu_{0}]\}.\label{eq:P1 def}
\end{equation}
It is clear that each ground state in $\mathcal{S}$ constitutes a
stable plateau in $\overline{\Omega}$, i.e., for every $\bm{s}\in\mathcal{S}$
there exists $\mathcal{P}\in\mathscr{P}^{1}$ such that $\bm{s}\in\mathcal{P}$. Note that here it is possible for a stable plateau to contain two or more ground states.
\begin{definition}[Cycle]
\label{def:cycle}A nonempty connected set $\mathcal{C}$ is a (nontrivial)
\emph{cycle} (cf. \cite[Definition 6.5]{OV}) if
\begin{equation}
\max_{\eta\in\mathcal{C}}\mathbb{H}(\eta)<\min_{\xi\in\partial\mathcal{C}}\mathbb{H}(\xi).\label{eq:cyc def}
\end{equation}
Each stable plateau is clearly a cycle. Then, the \emph{depth} of
cycle $\mathcal{C}$ is defined as
\begin{equation}
\Gamma^{\mathcal{C}}:=\min_{\partial\mathcal{C}}\mathbb{H}-\min_{\mathcal{C}}\mathbb{H}>0.\label{eq:depth-def}
\end{equation}
\end{definition}

Each cycle, in particular each stable plateau, is \emph{metastable}
in the sense that it takes an exponentially long time to exit (cf.
\cite[Theorem 6.23]{OV}) and that the exit time is asymptotically
exponentially distributed (cf. \cite[Theorem 6.30]{OV}). On the other
hand, any configuration that is not contained in a stable plateau
is not metastable, since its holding rate is of order $1$. Thus,
from now on we assume that $\nu_{0}\ge2$, consider $\mathscr{P}^{1}$
as the collection of all metastable elements and study the tunneling
transitions between the elements of $\mathscr{P}^{1}$.

It is easy to verify that for any two cycles $\mathcal{C},\mathcal{C}'$
with $\mathcal{C}\cap\mathcal{C}'\ne\emptyset$, it holds that $\mathcal{C}\subseteq\mathcal{C}'$
or $\mathcal{C}'\subseteq\mathcal{C}$ (see \cite[Proposition 6.8]{OV}
for a proof). Moreover, the following lemma holds.
\begin{lemma}
\label{lem:cyc bot P}For every cycle $\mathcal{C}$, its bottom $\mathcal{F}(\mathcal{C})$
is a union of stable plateaux.
\end{lemma}

\begin{proof}
Decompose $\mathcal{F}(\mathcal{C})$ into connected components as
$\mathcal{F}(\mathcal{C})=\mathcal{A}_{1}\cup\cdots\cup\mathcal{A}_{N}$
and consider $\partial\mathcal{A}_{n}$ for each $n\in[1,N]$. We
claim that $\mathbb{H}(\xi)>\mathbb{H}(\mathcal{A}_{n})$ for all
$\xi\in\partial\mathcal{A}_{n}$ which concludes the proof of the
lemma. Indeed, if $\xi\notin\mathcal{C}$ such that $\xi\in\partial\mathcal{C}$,
then $\mathbb{H}(\xi)>\mathbb{H}(\mathcal{A}_{n})$ by \eqref{eq:cyc def}.
If $\xi\in\mathcal{C}$, then $\mathbb{H}(\xi)>\mathbb{H}(\mathcal{A}_{n})$
since $\mathcal{A}_{n}\subseteq\mathcal{F}(\mathcal{C})$ and $\xi\in\mathcal{C}\setminus\mathcal{F}(\mathcal{C})$.
\end{proof}
Now, we present a general framework of constructing Markov jumps between
cycles.
\begin{definition}[Construction of Markovian jumps between cycles]
\label{def:gen const}Suppose that we are given a pair $(\mathscr{C},\Gamma^{\star})$
where $\mathscr{C}$ is a collection of disjoint cycles in $\overline{\Omega}$
with $|\mathscr{C}|\ge2$ and $\Gamma^{\star}$ is a positive real
number. Define
\begin{equation}
\mathscr{P}^{\mathscr{C}}:=\{\mathcal{F}(\mathcal{C}):\mathcal{C}\in\mathscr{C}\}.\label{eq:PC def}
\end{equation}
By Lemma \ref{lem:cyc bot P}, $\mathscr{P}^{\mathscr{C}}$ is a collection
of disjoint unions of stable plateaux in $\overline{\Omega}$. Define
\begin{equation}
\mathscr{C}^{\star}:=\{\mathcal{C}\in\mathscr{C}:\Gamma^{\mathcal{C}}\ge\Gamma^{\star}\}\quad\text{and}
\quad\mathscr{C}^{\sharp}:=\{\mathcal{C}\in\mathscr{C}:\Gamma^{\mathcal{C}}<\Gamma^{\star}\},\label{eq:C-star C-sharp def}
\end{equation}
such that $\mathscr{C}=\mathscr{C}^{\star}\cup\mathscr{C}^{\sharp}$.
In words, $\mathscr{C}^{\star}$ consists of the cycles in $\mathscr{C}$
with depth at least $\Gamma^{\star}$. Accordingly, define
\begin{equation}
\mathscr{P}^{\mathscr{C}^{\star}}:=\{\mathcal{F}(\mathcal{C}):\mathcal{C}\in\mathscr{C}^{\star}\}\quad\text{and}\quad\mathscr{P}^{\mathscr{C}^{\sharp}}:=\{\mathcal{F}(\mathcal{C}):\mathcal{C}\in\mathscr{C}^{\sharp}\},\label{eq:P-star P-sharp def}
\end{equation}
such that $\mathscr{P}^{\mathscr{C}}=\mathscr{P}^{\mathscr{C}^{\star}}\cup\mathscr{P}^{\mathscr{C}^{\sharp}}$.
In addition, define $\Delta^{\mathscr{C}}:=\overline{\Omega}\setminus\bigcup_{\mathcal{C}\in\mathscr{C}}\mathcal{C}$.

The \emph{contracted graph} $\Omega^{\mathscr{C}}$ is obtained from
$\overline{\Omega}$ by contracting the elements in each cycle $\mathcal{C}\in\mathscr{C}$
into single element $\mathcal{F}(\mathcal{C})\in\mathscr{P}^{\mathscr{C}}$,
such that $\Omega^{\mathscr{C}}=\Delta^{\mathscr{C}}\cup\mathscr{P}^{\mathscr{C}}$.
Then, the induced Markov chain $\{\mathfrak{X}^{\mathscr{C}}(t)\}_{t\ge0}$
in $\Omega^{\mathscr{C}}$ with transition rate $\mathfrak{R}^{\mathscr{C}}(\cdot,\cdot)$
is defined as follows:
\begin{equation}
\begin{cases}
\mathfrak{R}^{\mathscr{C}}(\eta,\xi):=1 & \text{if}\quad\eta\sim\xi,\quad\mathbb{H}(\xi)\le\mathbb{H}(\eta),\\
\mathfrak{R}^{\mathscr{C}}(\eta,\mathcal{F}(\mathcal{C})):=|\{\zeta\in\mathcal{C}:\eta\sim\zeta\}| & \text{if}\quad\eta\in\partial\mathcal{C},\\
\mathfrak{R}^{\mathscr{C}}(\mathcal{F}(\mathcal{C}),\eta):=|\mathcal{F}(\mathcal{C})|^{-1} |\{\zeta\in\mathcal{C}:\eta\sim\zeta\}| & \text{if}\quad\Gamma^{\mathcal{C}}\le\Gamma^{\star},\quad\eta\in\partial^{\star}\mathcal{C},
\end{cases}\label{eq:RC def}
\end{equation}
and $\mathfrak{R}^{\mathscr{C}}(\cdot,\cdot):=0$ otherwise. Note
that $\{\mathfrak{X}^{\mathscr{C}}(t)\}_{t\ge0}$ is not necessarily
irreducible, since every $\mathcal{F}(\mathcal{C})\in\mathscr{P}^{\mathscr{C}}$
with $\Gamma^{\mathcal{C}}>\Gamma^{\star}$ is an absorbing state.
Thus, $\{\mathfrak{X}^{\mathscr{C}}(t)\}_{t\ge0}$ encodes the asymptotic
jumps between the configurations and cycles with depth at most $\Gamma^{\star}$.
By \eqref{eq:C-star C-sharp def} and \eqref{eq:RC def}, among the
cycles in $\mathscr{C}^{\star}$ only those with depth exactly $\Gamma^{\star}$
have positive rates with respect to $\mathfrak{R}^{\mathscr{C}}(\cdot,\cdot)$.

Then, consider the \emph{trace} Markov chain $\{\mathfrak{X}^{\mathscr{C}^{\star}}(t)\}_{t\ge0}$
in $\mathscr{P}^{\mathscr{C}^{\star}}$ whose transition rate function
$\mathfrak{R}^{\mathscr{C}^{\star}}(\cdot,\cdot)$ is defined as\footnote{In this article, $\mathcal{T}_{\mathcal{A}}$ denotes the (random)
hitting time of set $\mathcal{A}$.}
\begin{equation}
\mathfrak{R}^{\mathscr{C}^{\star}}(\mathcal{F}(\mathcal{C}),\mathcal{F}(\mathcal{C}')):=\sum_{\eta\in\Delta^{\mathscr{C}}}\mathfrak{R}^{\mathscr{C}}(\mathcal{F}(\mathcal{C}),\eta) {\bf P}_{\eta}^{\mathscr{C}}[\mathcal{T}_{\mathcal{F}(\mathcal{C}')}=\mathcal{T}_{\mathscr{P}^{\mathscr{C}^{\star}}}],\label{eq:R-Cstar def}
\end{equation}
where ${\bf P}_{\eta}^{\mathscr{C}}$ denotes the law of $\{\mathfrak{X}^{\mathscr{C}}(t)\}_{t\ge0}$
starting from $\eta\in\Omega^{\mathscr{C}}$. Refer to \cite[Section 6.1]{BL10}
for the precise definition of the trace process.
\end{definition}

\subsection{\label{sec2.1}Hierarchical structure of stable plateaux}

\subsubsection*{Initial level}

Recall from \eqref{eq:P1 def} that $\mathscr{P}^{1}=\{\mathcal{P}_{1}^{1},\dots,\mathcal{P}_{\nu_{0}}^{1}\}$
is the collection of all $\nu_{0}\ge2$ stable plateaux in $\overline{\Omega}$.
For each $\mathcal{P}_{i}^{1}\in\mathscr{P}^{1}$, define the \emph{initial
depth} as
\begin{equation}
\Gamma_{i}^{1}:=\Phi(\mathcal{P}_{i}^{1},\breve{\mathcal{P}}_{i}^{1})-\mathbb{H}(\mathcal{P}_{i}^{1})>0,\quad\text{where}\quad\breve{\mathcal{P}}_{i}^{1}:=\bigcup_{j\in[1,\nu_{0}]:j\ne i}\mathcal{P}_{j}^{1}.\label{eq:Gammai1 def}
\end{equation}
Accordingly, define
\begin{equation}
\Gamma^{\star,1}:=\min_{i\in[1,\nu_{0}]}\Gamma_{i}^{1}>0.\label{eq:Gamma-star1 def}
\end{equation}
Then, define
\begin{equation}
\mathcal{V}_{i}^{1}:=\{\eta\in\Omega:\Phi(\mathcal{P}_{i}^{1},\eta)-\mathbb{H}(\mathcal{P}_{i}^{1})<\Gamma_{i}^{1}\}.\label{eq:Vi1 def}
\end{equation}

\begin{lemma}
\label{lem:Vi1 cycle}Collections $\mathcal{V}_{i}^{1}$ for $i\in[1,\nu_{0}]$
are disjoint cycles in $\overline{\Omega}$ with bottom $\mathcal{P}_{i}^{1}$
and depth $\Gamma_{i}^{1}$.
\end{lemma}

\begin{proof}
By \eqref{eq:Omega-bar def} and \eqref{eq:Gammai1 def} it is clear
that $\mathbb{H}(\mathcal{P}_{i}^{1})+\Gamma_{i}^{1}\le\overline{\Phi}$,
thus $\mathcal{V}_{i}^{1}$ is a subset of $\overline{\Omega}$.\footnote{The same reasoning applies to all other collections in the remainder
as well; thus, we regard them to be subsets of $\overline{\Omega}$
without further explanation.} To prove the disjointness, suppose the contrary that there exists
$\eta\in\mathcal{V}_{i}^{1}\cap\mathcal{V}_{j}^{1}$ for $i\ne j$.
By \eqref{eq:Phi max ineq},
\[
\Phi(\mathcal{P}_{i}^{1},\mathcal{P}_{j}^{1})\le\max\{\Phi(\mathcal{P}_{i}^{1},\eta),\Phi(\mathcal{P}_{j}^{1},\eta)\}<\max\{\mathbb{H}(\mathcal{P}_{i}^{1})+\Gamma_{i}^{1},\mathbb{H}(\mathcal{P}_{j}^{1})+\Gamma_{j}^{1}\}.
\]
On the other hand, by \eqref{eq:Gammai1 def}, $\mathbb{H}(\mathcal{P}_{i}^{1})+\Gamma_{i}^{1}=\Phi(\mathcal{P}_{i}^{1},\breve{\mathcal{P}}_{i}^{1})\le\Phi(\mathcal{P}_{i}^{1},\mathcal{P}_{j}^{1})$
and similarly $\mathbb{H}(\mathcal{P}_{j}^{1})+\Gamma_{j}^{1}\le\Phi(\mathcal{P}_{j}^{1},\mathcal{P}_{i}^{1})$,
which contradict the displayed inequality. Thus, the collections $\mathcal{V}_{i}^{1}$
for $i\in[1,\nu_{0}]$ are disjoint.

By \eqref{eq:Vi1 def} and the fact that $\mathcal{P}_{i}^{1}$ is
connected, it follows immediately that $\mathcal{V}_{i}^{1}$ is also
connected. Moreover, by the definition of $\mathcal{V}_{i}^{1}$,
it is clear that $\mathbb{H}(\eta)<\mathbb{H}(\mathcal{P}_{i}^{1})+\Gamma_{i}^{1}$
for all $\eta\in\mathcal{V}_{i}^{1}$ and $\mathbb{H}(\zeta)\ge\mathbb{H}(\mathcal{P}_{i}^{1})+\Gamma_{i}^{1}$
for all $\zeta\in\partial\mathcal{V}_{i}^{1}$. These facts prove
that $\mathcal{V}_{i}^{1}$ is a cycle.

Since $\mathcal{F}(\mathcal{V}_{i}^{1})$ is a union of stable plateaux
by Lemma \ref{lem:cyc bot P} and $\mathcal{V}_{i}^{1}\cap\mathcal{P}_{j}^{1}=\emptyset$
for all $j\ne i$ by the disjointness, we obtain that $\mathcal{F}(\mathcal{V}_{i}^{1})=\mathcal{P}_{i}^{1}$.
Finally, there exists $\xi\in\partial\mathcal{V}_{i}^{1}$ such that
$\mathbb{H}(\xi)=\mathbb{H}(\mathcal{P}_{i}^{1})+\Gamma_{i}^{1}$
due to the existence of a path $\mathcal{P}_{i}^{1}\to\breve{\mathcal{P}}_{i}^{1}$
of height $\mathbb{H}(\mathcal{P}_{i}^{1})+\Gamma_{i}^{1}$ guaranteed
by \eqref{eq:Gammai1 def}. Collecting these observations, we calculate
as
\[
\Gamma^{\mathcal{V}_{i}^{1}}=\min_{\partial\mathcal{V}_{i}^{1}}\mathbb{H}-\min_{\mathcal{V}_{i}^{1}}\mathbb{H}=(\mathbb{H}(\mathcal{P}_{i}^{1})+\Gamma_{i}^{1})-\mathbb{H}(\mathcal{P}_{i}^{1})=\Gamma_{i}^{1}.
\]
This completes the proof of Lemma \ref{lem:Vi1 cycle}.
\end{proof}
Collect
\begin{equation}
\mathscr{C}^{1}:=\{\mathcal{V}_{i}^{1}:i\in[1,\nu_{0}]\}.\label{eq:C1 def}
\end{equation}
We apply the general construction given in Definition \ref{def:gen const}
to $(\mathscr{C}^{1},\Gamma^{\star,1})$. By \eqref{eq:PC def}
and Lemma \ref{lem:Vi1 cycle}, it readily holds that $\mathscr{P}^{\mathscr{C}^{1}}=\mathscr{P}^{1}$.
\begin{definition}
\label{nota:abbr 1}We abbreviate $\mathscr{C}^{\star,1}:=(\mathscr{C}^{1})^{\star}$,
$\mathscr{C}^{\sharp,1}:=(\mathscr{C}^{1})^{\sharp}$, $\mathscr{P}^{\star,1}:=\mathscr{P}^{(\mathscr{C}^{1})^{\star}}$,\\
$\mathscr{P}^{\sharp,1}:=\mathscr{P}^{(\mathscr{C}^{1})^{\sharp}}$,
$\Delta^{1}:=\Delta^{\mathscr{C}^{1}}$, $\Omega^{1}:=\Omega^{\mathscr{C}^{1}}$,
$\mathfrak{X}^{1}(t):=\mathfrak{X}^{\mathscr{C}^{1}}(t)$, $\mathfrak{R}^{1}(\cdot,\cdot):=\mathfrak{R}^{\mathscr{C}^{1}}(\cdot,\cdot)$,\\
$\mathfrak{X}^{\star,1}(t):=\mathfrak{X}^{(\mathscr{C}^{1})^{\star}}(t)$,
and $\mathfrak{R}^{\star,1}(\cdot,\cdot):=\mathfrak{R}^{(\mathscr{C}^{1})^{\star}}(\cdot,\cdot)$.
\end{definition}

Note that by \eqref{eq:Gamma-star1 def} and Lemma \ref{lem:Vi1 cycle},
$\Gamma^{\mathcal{V}_{i}^{1}}\ge\Gamma^{\star,1}$ for all $i\in[1,\nu_{0}]$,
thus it holds that
\begin{equation}
\mathscr{C}^{\star,1}=\mathscr{C}^{1}\quad\text{and}\quad\mathscr{C}^{\sharp,1}=\emptyset,\label{eq:C-star1 C-sharp1 prop}
\end{equation}
and accordingly,
\begin{equation}
\mathscr{P}^{\star,1}=\mathscr{P}^{1}\quad\text{and}\quad\mathscr{P}^{\sharp,1}=\emptyset.\label{eq:P-star1 P-sharp1 prop}
\end{equation}
Now, $\{\mathfrak{X}^{\star,1}(t)\}_{t\ge0}$ decomposes $\mathscr{P}^{\star,1}$
into
\begin{equation}
\mathscr{P}^{\star,1}=\mathscr{P}_{1}^{\star,1}\cup\cdots\cup\mathscr{P}_{\nu_{1}}^{\star,1}\cup\mathscr{P}_{{\rm tr}}^{\star,1},\label{eq:P-star1 dec}
\end{equation}
where $\mathscr{P}_{1}^{\star,1},\dots,\mathscr{P}_{\nu_{1}}^{\star,1}$
are irreducible components and $\mathscr{P}_{{\rm tr}}^{\star,1}$
is the collection of transient elements. Accordingly, decompose
\begin{equation}
\mathscr{C}^{\star,1}=\mathscr{C}_{1}^{\star,1}\cup\cdots\cup\mathscr{C}_{\nu_{1}}^{\star,1}\cup\mathscr{C}_{{\rm tr}}^{\star,1},\label{eq:C-star1 dec}
\end{equation}
where $\mathscr{C}_{m}^{\star,1}:=\{\mathcal{V}_{i}^{1}\in\mathscr{C}^{\star,1}:\mathcal{P}_{i}^{1}\in\mathscr{P}_{m}^{\star,1}\}$
for $m\in\{1,\dots,\nu_{1},{\rm tr}\}$.

\subsubsection*{From level $\mathfrak{h}-1$ to level $\mathfrak{h}$}

For an integer $\mathfrak{h}\ge2$, suppose that we are given a collection $\mathscr{C}^{\mathfrak{h}-1}$
of disjoint cycles in $\overline{\Omega}$ and a collection $\mathscr{P}^{\mathfrak{h}-1}=\{\mathcal{F}(\mathcal{C}):\mathcal{C}\in\mathscr{C}^{\mathfrak{h}-1}\}$,
along with decompositions
\begin{equation}
\mathscr{C}^{\mathfrak{h}-1}=\bigcup_{m\in[1,\nu_{\mathfrak{h}-1}]}\mathscr{C}_{m}^{\star,\mathfrak{h}-1}\cup\mathscr{C}_{{\rm tr}}^{\star,\mathfrak{h}-1}\cup\mathscr{C}^{\sharp,\mathfrak{h}-1}\label{eq:Ch-1 Ph-1 ind hyp}
\end{equation}
and
\[
\mathscr{P}^{\mathfrak{h}-1}=\bigcup_{m\in[1,\nu_{\mathfrak{h}-1}]}\mathscr{P}_{m}^{\star,\mathfrak{h}-1}\cup\mathscr{P}_{{\rm tr}}^{\star,\mathfrak{h}-1}\cup\mathscr{P}^{\sharp,\mathfrak{h}-1},
\]
where $\mathscr{P}_{m}^{\star,\mathfrak{h}-1}=\{\mathcal{F}(\mathcal{C}):\mathcal{C}\in\mathscr{C}_{m}^{\star,\mathfrak{h}-1}\}$
for $m\in\{1,\dots,\nu_{\mathfrak{h}-1},{\rm tr}\}$ and $\mathscr{P}^{\sharp,\mathfrak{h}-1}=\{\mathcal{F}(\mathcal{C}):\mathcal{C}\in\mathscr{C}^{\sharp,\mathfrak{h}-1}\}$.
Note that Lemma \ref{lem:Vi1 cycle}, \eqref{eq:C-star1 C-sharp1 prop},
\eqref{eq:P-star1 P-sharp1 prop}, \eqref{eq:P-star1 dec} and \eqref{eq:C-star1 dec}
guarantee this assumption for $h=2$. Further suppose that $\nu_{\mathfrak{h}-1}\ge2$.
Then, for each $i\in[1,\nu_{\mathfrak{h}-1}]$ define
\begin{equation}
\mathcal{P}_{i}^{\mathfrak{h}}:=\bigcup_{\mathcal{P}\in\mathscr{P}_{i}^{\star,\mathfrak{h}-1}}\mathcal{P}\quad\text{and}\quad\mathscr{P}^{\star,\mathfrak{h}}:=\{\mathcal{P}_{i}^{\mathfrak{h}}:i\in[1,\nu_{\mathfrak{h}-1}]\}.\label{eq:P-starh def}
\end{equation}

\begin{lemma}
\label{lem:Pih energy}For each $i\in[1,\nu_{\mathfrak{h}-1}]$, it holds that
$\mathbb{H}(\eta)=\mathbb{H}(\xi)$ for all $\eta,\xi\in\mathcal{P}_{i}^{\mathfrak{h}}$.
\end{lemma}

As in \eqref{eq:Gammai1 def}, define the \emph{$\mathfrak{h}$-th depth} as
\begin{equation}
\Gamma_{i}^{\mathfrak{h}}:=\Phi(\mathcal{P}_{i}^{\mathfrak{h}},\breve{\mathcal{P}}_{i}^{\mathfrak{h}})-\mathbb{H}(\mathcal{P}_{i}^{\mathfrak{h}})>0,\quad\text{where}\quad\breve{\mathcal{P}}_{i}^{\mathfrak{h}}:=\bigcup_{j\in[1,\nu_{\mathfrak{h}-1}]:j\ne i}\mathcal{P}_{j}^{\mathfrak{h}}.\label{eq:Gammaih def}
\end{equation}
Then, write
\begin{equation}
\Gamma^{\star,\mathfrak{h}}:=\min_{i\in[1,\nu_{\mathfrak{h}-1}]}\Gamma_{i}^{\mathfrak{h}}>0.\label{eq:Gamma-starh def}
\end{equation}

\begin{lemma}
\label{lem:Gamma-star inc}It holds that $\Gamma^{\star,\mathfrak{h}}>\Gamma^{\star,\mathfrak{h}-1}$.
\end{lemma}

Next, define
\begin{equation}
\mathcal{V}_{i}^{\mathfrak{h}}:=\{\eta\in\Omega:\Phi(\mathcal{P}_{i}^{\mathfrak{h}},\eta)-\mathbb{H}(\mathcal{P}_{i}^{\mathfrak{h}})<\Gamma_{i}^{\mathfrak{h}}\}.\label{eq:Vih def}
\end{equation}

\begin{lemma}
\label{lem:Vih cycle}Collections $\mathcal{V}_{i}^{\mathfrak{h}}$ for $i\in[1,\nu_{\mathfrak{h}-1}]$
are disjoint cycles in $\overline{\Omega}$ with bottom $\mathcal{P}_{i}^{\mathfrak{h}}$
and depth $\Gamma_{i}^{\mathfrak{h}}$.
\end{lemma}

Now, define
\begin{equation}
\mathscr{C}^{\mathfrak{h}}:=\{\mathcal{V}_{i}^{\mathfrak{h}}:i\in[1,\nu_{\mathfrak{h}-1}]\}\cup\{\mathcal{C}\in\mathscr{C}_{{\rm tr}}^{\star,\mathfrak{h}-1}\cup\mathscr{C}^{\sharp,\mathfrak{h}-1}:\mathcal{C}\cap\mathcal{V}_{i}^{\mathfrak{h}}=\emptyset,\ \forall i\in[1,\nu_{\mathfrak{h}-1}]\}.\label{eq:Ch def}
\end{equation}
Then, we apply the construction in Definition \ref{def:gen const}
to $(\mathscr{C}^{\mathfrak{h}},\Gamma^{\star,\mathfrak{h}})$.
\begin{lemma}
\label{lem:Ch prop}It holds that $(\mathscr{C}^{\mathfrak{h}})^{\star}=\{\mathcal{V}_{i}^{\mathfrak{h}}:i\in[1,\nu_{\mathfrak{h}-1}]\}$
and
\[
(\mathscr{C}^{\mathfrak{h}})^{\sharp}=\{\mathcal{C}\in\mathscr{C}_{{\rm tr}}^{\star,\mathfrak{h}-1}\cup\mathscr{C}^{\sharp,\mathfrak{h}-1}:\mathcal{C}\cap\mathcal{V}_{i}^{\mathfrak{h}}=\emptyset, \  \forall i\in[1,\nu_{\mathfrak{h}-1}]\}.
\]
\end{lemma}

By Lemmas \ref{lem:Vih cycle} and \ref{lem:Ch prop}, we have $\mathscr{P}^{(\mathscr{C}^{\mathfrak{h}})^{\star}}=\{\mathcal{P}_{i}^{\mathfrak{h}}:i\in[1,\nu_{\mathfrak{h}-1}]\}=\mathscr{P}^{\star,\mathfrak{h}}$.
\begin{definition}
\label{nota:abbr h}Abbreviate $\mathscr{P}^{\mathfrak{h}}:=\mathscr{P}^{\mathscr{C}^{\mathfrak{h}}}$,
$\mathscr{C}^{\star,\mathfrak{h}}:=(\mathscr{C}^{\mathfrak{h}})^{\star}$, $\mathscr{C}^{\sharp,\mathfrak{h}}:=(\mathscr{C}^{\mathfrak{h}})^{\sharp}$,\\
$\mathscr{P}^{\sharp,\mathfrak{h}}:=\mathscr{P}^{(\mathscr{C}^{\mathfrak{h}})^{\sharp}}$,
$\Delta^{\mathfrak{h}}:=\Delta^{\mathscr{C}^{\mathfrak{h}}}$, $\Omega^{\mathfrak{h}}:=\Omega^{\mathscr{C}^{\mathfrak{h}}}$,
$\mathfrak{X}^{\mathfrak{h}}(t):=\mathfrak{X}^{\mathscr{C}^{\mathfrak{h}}}(t)$,\\
$\mathfrak{R}^{\mathfrak{h}}(\cdot,\cdot):=\mathfrak{R}^{\mathscr{C}^{\mathfrak{h}}}(\cdot,\cdot)$,
$\mathfrak{X}^{\star,\mathfrak{h}}(t):=\mathfrak{X}^{(\mathscr{C}^{\mathfrak{h}})^{\star}}(t)$,
and $\mathfrak{R}^{\star,\mathfrak{h}}(\cdot,\cdot):=\mathfrak{R}^{(\mathscr{C}^{\mathfrak{h}})^{\star}}(\cdot,\cdot)$.
\end{definition}

The Markov chain $\{\mathfrak{X}^{\star,\mathfrak{h}}(t)\}_{t\ge0}$ decomposes
$\mathscr{P}^{\star,\mathfrak{h}}$ into
\begin{equation}
\mathscr{P}^{\star,\mathfrak{h}}=\mathscr{P}_{1}^{\star,\mathfrak{h}}\cup\cdots\cup\mathscr{P}_{\nu_{\mathfrak{h}}}^{\star,\mathfrak{h}}\cup\mathscr{P}_{{\rm tr}}^{\star,\mathfrak{h}},\label{eq:P-starh dec}
\end{equation}
with $\nu_{\mathfrak{h}}$ irreducible components and a transient collection.
Accordingly, we obtain
\begin{equation}
\mathscr{C}^{\star,\mathfrak{h}}=\mathscr{C}_{1}^{\star,\mathfrak{h}}\cup\cdots\cup\mathscr{C}_{\nu_{\mathfrak{h}}}^{\star,\mathfrak{h}}\cup\mathscr{C}_{{\rm tr}}^{\star,\mathfrak{h}},\label{eq:C-starh dec}
\end{equation}
where $\mathscr{C}_{m}^{\star,\mathfrak{h}}:=\{\mathcal{V}_{i}^{\mathfrak{h}}\in\mathscr{C}^{\star,\mathfrak{h}}:\mathcal{P}_{i}^{\mathfrak{h}}\in\mathscr{P}_{m}^{\star,\mathfrak{h}}\}$
for $m\in\{1,\dots,\nu_{\mathfrak{h}},{\rm tr}\}$.

Finally, we may repeat the same inductive procedure provided that
$\nu_{\mathfrak{h}}\ge2$.
According to the construction, the number of irreducible components
decreases strictly:
\begin{theorem}
\label{thm:nu dec}For all $\mathfrak{h}\ge1$, it holds that $\nu_{\mathfrak{h}}<\nu_{\mathfrak{h}-1}$.
\end{theorem}

\begin{figure}[t]\centering
{\begin{tikzpicture}[scale=0.7]
\draw (0,3) sin (0.5,0);
\draw (0.5,0) cos (0.75,0.5); \draw (0.75,0.5) sin (1,1);
\draw (1,1) cos (1.25,0.5); \draw (1.25,0.5) sin (1.5,0);
\draw (1.5,0) cos (1.75,1); \draw (1.75,1) sin (2,2);
\draw (2,2) cos (2.25,1); \draw (2.25,1) sin (2.5,0);
\draw (2.5,0) cos (2.75,1.5); \draw (2.75,1.5) sin (3,3);
\draw (3,3) cos (3.25,2.5); \draw (3.25,2.5) sin (3.5,2);
\draw (3.5,2) cos (3.75,2.5); \draw (3.75,2.5) sin (4,3);
\draw (4,3) cos (4.25,2); \draw (4.25,2) sin (4.5,1);
\draw (4.5,1) cos (4.75,1.5); \draw (4.75,1.5) sin (5,2);
\draw (5,2) cos (5.25,1.5); \draw (5.25,1.5) sin (5.5,1);
\draw (5.5,1) cos (5.75,2); \draw (5.75,2) sin (6,3);
\draw (6,3) cos (6.25,1.5); \draw (6.25,1.5) sin (6.5,0);
\draw (6.5,0) cos (7,3);

\foreach \i in {0.5,1.5,2.5,6.5} {
\draw[very thick] (\i-0.1,0)--(\i+0.1,0); }
\foreach \i in {4.5,5.5} {
\draw[very thick] (\i-0.1,1)--(\i+0.1,1); }
\foreach \i in {3.5} {
\draw[very thick] (\i-0.1,2)--(\i+0.1,2); }

\draw (0.5,0) node[below]{$\bm{\mathcal{P}_1^1}$};
\draw (1.5,0) node[below]{$\bm{\mathcal{P}_2^1}$};
\draw (2.5,0) node[below]{$\bm{\mathcal{P}_3^1}$};
\draw (3.5,2) node[below]{$\mathcal{P}_4^1$};
\draw (4.5,1) node[below]{$\bm{\mathcal{P}_5^1}$};
\draw (5.5,1) node[below]{$\bm{\mathcal{P}_6^1}$};
\draw (6.5,0) node[below]{$\bm{\mathcal{P}_7^1}$};

\draw[red,<-] (0.5,0.75) sin (1,1.25); \draw[red,<-] (1.5,0.75) sin (1,1.25);
\draw[red,<-] (2.5,2.75) sin (3,3.25); \draw[red] (3.5,2.75) sin (3,3.25);
\draw[red] (3.5,2.75) sin (4,3.25); \draw[red,<-] (4.5,2.75) sin (4,3.25);
\draw[red,<-] (4.5,1.75) sin (5,2.25); \draw[red,<-] (5.5,1.75) sin (5,2.25);
\end{tikzpicture}
\hspace{15mm}
\begin{tikzpicture}[scale=0.7]
\draw (0,3) sin (0.5,0);
\draw (0.5,0) cos (0.75,0.5); \draw (0.75,0.5) sin (1,1);
\draw (1,1) cos (1.25,0.5); \draw (1.25,0.5) sin (1.5,0);
\draw (1.5,0) cos (1.75,1); \draw (1.75,1) sin (2,2);
\draw (2,2) cos (2.25,1); \draw (2.25,1) sin (2.5,0);
\draw (2.5,0) cos (2.75,1.5); \draw (2.75,1.5) sin (3,3);
\draw (3,3) cos (3.25,2.5); \draw (3.25,2.5) sin (3.5,2);
\draw (3.5,2) cos (3.75,2.5); \draw (3.75,2.5) sin (4,3);
\draw (4,3) cos (4.25,2); \draw (4.25,2) sin (4.5,1);
\draw (4.5,1) cos (4.75,1.5); \draw (4.75,1.5) sin (5,2);
\draw (5,2) cos (5.25,1.5); \draw (5.25,1.5) sin (5.5,1);
\draw (5.5,1) cos (5.75,2); \draw (5.75,2) sin (6,3);
\draw (6,3) cos (6.25,1.5); \draw (6.25,1.5) sin (6.5,0);
\draw (6.5,0) cos (7,3);

\draw[very thick] (0.4,0)--(1.6,0);
\draw[very thick] (2.4,0)--(2.6,0);
\draw[very thick] (4.4,1)--(5.6,1);
\draw[very thick] (6.4,0)--(6.6,0);

\draw (1,0) node[below]{$\bm{\mathcal{P}_1^2}$};
\draw (2.5,0) node[below]{$\bm{\mathcal{P}_2^2}$};
\draw (5,1) node[below]{$\mathcal{P}_3^2$};
\draw (6.5,0) node[below]{$\bm{\mathcal{P}_4^2}$};

\draw[red,<-] (1,1.75) sin (1.75,2.25); \draw[red,<-] (2.5,1.75) sin (1.75,2.25);
\draw[red,<-] (2.5,2.75) sin (3.75,3.25); \draw[red] (5,2.75) sin (3.75,3.25);
\draw[red] (5,2.75) sin (5.75,3.25); \draw[red,<-] (6.5,2.75) sin (5.75,3.25);
\end{tikzpicture}\\
\begin{tikzpicture}[scale=0.7]
\draw (0,3) sin (0.5,0);
\draw (0.5,0) cos (0.75,0.5); \draw (0.75,0.5) sin (1,1);
\draw (1,1) cos (1.25,0.5); \draw (1.25,0.5) sin (1.5,0);
\draw (1.5,0) cos (1.75,1); \draw (1.75,1) sin (2,2);
\draw (2,2) cos (2.25,1); \draw (2.25,1) sin (2.5,0);
\draw (2.5,0) cos (2.75,1.5); \draw (2.75,1.5) sin (3,3);
\draw (3,3) cos (3.25,2.5); \draw (3.25,2.5) sin (3.5,2);
\draw (3.5,2) cos (3.75,2.5); \draw (3.75,2.5) sin (4,3);
\draw (4,3) cos (4.25,2); \draw (4.25,2) sin (4.5,1);
\draw (4.5,1) cos (4.75,1.5); \draw (4.75,1.5) sin (5,2);
\draw (5,2) cos (5.25,1.5); \draw (5.25,1.5) sin (5.5,1);
\draw (5.5,1) cos (5.75,2); \draw (5.75,2) sin (6,3);
\draw (6,3) cos (6.25,1.5); \draw (6.25,1.5) sin (6.5,0);
\draw (6.5,0) cos (7,3);

\draw[very thick] (0.4,0)--(2.6,0);
\draw[very thick] (6.4,0)--(6.6,0);

\draw (1.5,0) node[below]{$\bm{\mathcal{P}_1^3}$};
\draw (6.5,0) node[below]{$\bm{\mathcal{P}_2^3}$};

\draw[red,<-] (1.5,2.75) sin (4,3.75); \draw[red,<-] (6.5,2.75) sin (4,3.75);
\end{tikzpicture}}\caption{\label{fig1}Example of a hierarchical decomposition of stable plateaux
in $\mathscr{P}^{1}$ with $\mathfrak{m}=3$. At each level $h\in[1,3]$,
bold-faced elements are recurrent and the rest are transient with
respect to $\{\mathfrak{X}^{\star,h}(t)\}_{t\ge0}$. At level $1$,
we have $\mathscr{P}_{1}^{\star,1}=\{\mathcal{P}_{1}^{1},\mathcal{P}_{2}^{1}\}$,
$\mathscr{P}_{2}^{\star,1}=\{\mathcal{P}_{3}^{1}\}$, $\mathscr{P}_{3}^{\star,1}=\{\mathcal{P}_{5}^{1},\mathcal{P}_{6}^{1}\}$,
$\mathscr{P}_{4}^{\star,1}=\{\mathcal{P}_{7}^{1}\}$ and $\mathscr{P}_{{\rm tr}}^{\star,1}=\{\mathcal{P}_{4}^{1}\}$.
At level $2$, we have $\mathscr{P}_{1}^{\star,2}=\{\mathcal{P}_{1}^{2},\mathcal{P}_{2}^{2}\}$,
$\mathscr{P}_{2}^{\star,2}=\{\mathcal{P}_{4}^{2}\}$ and $\mathscr{P}_{{\rm tr}}^{\star,2}=\{\mathcal{P}_{3}^{2}\}$.
Finally, at level $\mathfrak{m}=3$, we have $\mathscr{P}^{\star,3}=\mathscr{P}_{1}^{\star,3}=\{\mathcal{P}_{1}^{3},\mathcal{P}_{2}^{3}\}$
which is exactly composed of the ground states.}
\end{figure}
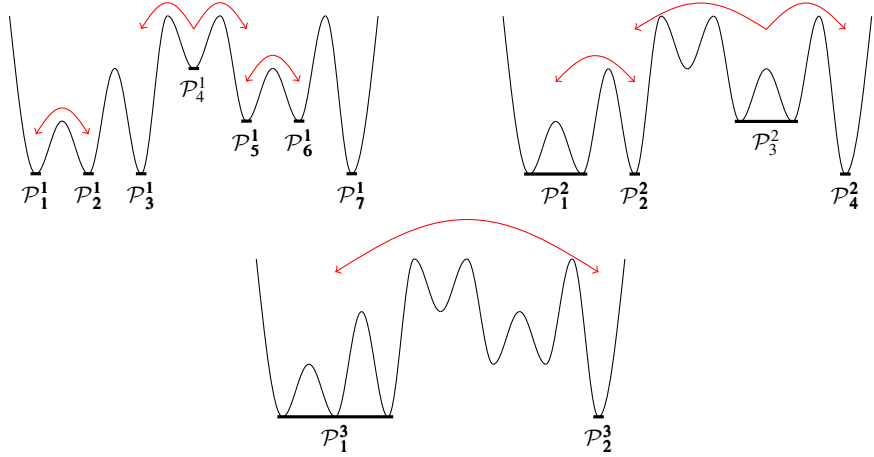

By Theorem \ref{thm:nu dec}, there exists a terminal integer $\mathfrak{m}$
such that $\nu_{\mathfrak{m}-1}>1$ and $\nu_{\mathfrak{m}}=1$. In
turn, the sequence $\mathscr{P}^{1},\mathscr{P}^{2},\dots,\mathscr{P}^{\mathfrak{m}}$
constitutes the full hierarchical decomposition of the stable plateaux
in $\overline{\Omega}$. See Figure \ref{fig1} for an illustration.

In addition, the ground states in $\mathcal{S}$ are always contained
in the recurrent collection: for each $\mathfrak{h}\ge1$, define
\begin{equation}
\mathscr{P}_{{\rm rec}}^{\star,\mathfrak{h}}:=\mathscr{P}_{1}^{\star,\mathfrak{h}}\cup\cdots\cup\mathscr{P}_{\nu_{\mathfrak{h}}}^{\star,\mathfrak{h}}\quad\text{and}\quad\mathscr{C}_{{\rm rec}}^{\star,\mathfrak{h}}:=\mathscr{C}_{1}^{\star,\mathfrak{h}}\cup\cdots\cup\mathscr{C}_{\nu_{\mathfrak{h}}}^{\star,\mathfrak{h}}.\label{eq:P-starh C-starh rec}
\end{equation}

\begin{theorem}[Ground states are always recurrent]
\label{thm:ground states rec}For all $\bm{s}\in\mathcal{S}$ and
$\mathfrak{h}\ge1$, there exists $\mathcal{P}_{i}^{\mathfrak{h}}\in\mathscr{P}_{{\rm rec}}^{\star,\mathfrak{h}}$
such that $\bm{s}\in\mathcal{P}_{i}^{\mathfrak{h}}$.
\end{theorem}

In particular, by Theorem \ref{thm:ground states rec} with $\mathfrak{h}=\mathfrak{m}$
and Lemma \ref{lem:Pih energy}, the unique irreducible collection
$\mathscr{P}_{1}^{\star,\mathfrak{m}}$ of the terminal level $\mathfrak{m}$
consists of exactly all ground states in $\mathcal{S}$.

\subsection{\label{sec2.2}Metastable hierarchy of tunneling transitions between
stable plateaux}

For each $\mathfrak{h}\in[1,\mathfrak{m}]$, define
\begin{equation}
\mathcal{V}^{\star,\mathfrak{h}}:=\bigcup_{i\in[1,\nu_{\mathfrak{h}-1}]}\mathcal{V}_{i}^{\mathfrak{h}}.\label{eq:V-starh def}
\end{equation}
Define a projection function $\Psi^{\mathfrak{h}}:\mathcal{V}^{\star,\mathfrak{h}}\to\mathscr{P}^{\star,\mathfrak{h}}$
as
\begin{equation}
\Psi^{\mathfrak{h}}(\eta):=\mathcal{P}_{i}^{\mathfrak{h}}\quad\text{for each}\quad\eta\in\mathcal{V}_{i}^{\mathfrak{h}}\quad\text{and}\quad i\in[1,\nu_{\mathfrak{h}-1}].\label{eq:Psih def}
\end{equation}
Consider the trace process $\{\eta_{\beta}^{\mathfrak{h}}(t)\}_{t\ge0}$ of the
original process in $\mathcal{V}^{\star,\mathfrak{h}}$. Then, define the \emph{$\mathfrak{h}$-th
order process} $\{X_{\beta}^{\mathfrak{h}}(t)\}_{t\ge0}$ in $\mathscr{P}^{\star,\mathfrak{h}}$
as 
\[
X_{\beta}^{\mathfrak{h}}(t):=\Psi^{\mathfrak{h}}(\eta_{\beta}^{\mathfrak{h}}(e^{\Gamma^{\star,\mathfrak{h}}\beta}t))\quad\text{for}\quad t\ge0.
\]
We are ready to state our main result.
\begin{theorem}[Hierarchical tunneling metastable transitions]
\label{thm:main}For each $\mathfrak{h}\in[1,\mathfrak{m}]$, the following
statements are valid.
\begin{enumerate}
\item The $\mathfrak{h}$-th order process $\{X_{\beta}^{\mathfrak{h}}(t)\}_{t\ge0}$ in $\mathscr{P}^{\star,\mathfrak{h}}$
converges to $\{\mathfrak{X}^{\star,\mathfrak{h}}(t)\}_{t\ge0}$.
\item The original process spends negligible time outside $\mathcal{V}^{\star,\mathfrak{h}}$
in the time scale $e^{\Gamma^{\star,\mathfrak{h}}\beta}$:
\[
\lim_{\beta\to\infty}\mathbb{E}_{\eta}\Big[\int_{0}^{T}{\bf 1}\{\eta_{\beta}(e^{\Gamma^{\star,\mathfrak{h}}\beta}t)\notin\mathcal{V}^{\star,\mathfrak{h}}\}{\rm d}t\Big]=0\quad\text{for all}\quad T>0\quad\text{and}\quad\eta\in\mathcal{V}^{\star,\mathfrak{h}}.
\]
\end{enumerate}
\end{theorem}

According to Theorem \ref{thm:main}, \eqref{eq:P-star P-sharp def}
and \eqref{eq:RC def}, at each level $\mathfrak{h}\in[1,\mathfrak{m}]$, there
exist tunneling transitions between the stable plateaux of depth at
least $\Gamma^{\star,\mathfrak{h}}$ in the time scale $e^{\beta\Gamma^{\star,\mathfrak{h}}}$,
where those with depth strictly greater than $\Gamma^{\star,\mathfrak{h}}$ are
absorbing. This is consistent with \eqref{eq:Gamma-starh def} in
that $\Gamma^{\star,\mathfrak{h}}$ is the minimum energy barrier between the
stable plateaux at level $\mathfrak{h}$.

\section{\label{sec3}Applications: Glauber/Kawasaki Dynamics on the Ising Model}

In this section, we apply the general results presented in Section
\ref{sec2} to concrete examples. We fix the Ising model as our configuration system
and consider four different types of dynamics defined thereon.

\subsection*{Ising model}

For fixed positive integers $K$ and $L$, we consider a two-dimensional
periodic square lattice $\Lambda=(V,E)$ of side lengths $K$ and
$L$, i.e.,
\[
V:=\mathbb{T}_{K}\times\mathbb{T}_{L}=\{0,1,\dots,K-1\}\times\{0,1,\dots,L-1\},
\]
and $E$ is the set of unordered nearest-neighbor bonds in $V$. Assume that $K \ge L$.
The corresonding configuration space $\Omega$ is defined as
\[
\Omega:=\{\eta=(\eta(x))_{x\in V}\in\{-1,+1\}^{V} \}.
\]
In the spin system point of view, the values $\pm1$ correspond to two different types of spins
that exist in the system; whereas, in the particle (or gas) system point of view, the value $+1$ (resp. $-1$)
indicates the occupied (resp. vacant) state of each site. In this regard, in the literature,
the possible spin values are sometimes selected as $0$ (vacant) or $1$ (occupied), instead of $\pm1$
as in the present article.

For each $\eta\in\Omega$, we define the Hamiltonian $\mathbb{H}(\eta)$
as
\begin{equation}
\mathbb{H}(\eta):=- \frac12 \sum_{\{x,y\}\in E}\eta(x)\eta(y) - \frac{h}2 \sum_{x \in V} {\bf 1} \{ \eta(x) = +1 \} .\label{eq:H def}
\end{equation}
Here, $h \ge 0$ denotes the external field towards spin $+1$.
According to this Hamiltonian, we assign a Gibbs measure
to $\Omega$ as in \eqref{eq:Gibbs def}.

\subsection*{Glauber dynamics}

The \emph{Glauber dynamics} in $\Omega$ is the continuous-time Markov chain $\{\eta_\beta^{\rm Gl}(t)\}_{t\ge0}$ whose transition rate function $r_\beta^{\rm Gl} : \Omega\times \Omega \to [0,\infty)$ is defined as
\begin{equation}\label{eq:Gl}
r_\beta^{\rm Gl}(\eta,\xi) := \begin{cases}
e^{-\beta \max\{ \mathbb{H}(\eta^x)-\mathbb{H}(x),0\}} & \text{if} \quad \xi = \eta^x \quad \text{and} \quad x \in V,\\
0 & \text{otherwise},
\end{cases}
\end{equation}
where $\eta^x$ is a new configuration obtained from $\eta$ by flipping the spin at site $x$:
\[
\eta^x(x)=-\eta(x)\quad \text{and} \quad \eta^x(y)=\eta(y) \quad \text{for} \quad y \ne x.
\]
It is routine to check that the Glauber transition rate is of form \eqref{eq:rbeta def}, and that the resulting graph structure on $\Omega$ is connected.

\subsection{\label{sec3.1}Case 1: Glauber dynamics, positive external field}

In this subsection, we assume that $h>0$ in \eqref{eq:H def} and consider the Glauber dynamics defined by \eqref{eq:Gl}. This case was one of the first models that exhibit nontrivial and complicated structure of metastability phenomenon, studied from the 80s \cite{NS91,NS92}. The notation and results summarized in this subsection are mainly taken from \cite{BL11}.

To avoid technical issues, we assume that $0<h<2$, $2/h \notin \mathbb{N}$, and $L> (\lfloor 2/h \rfloor +1)^2 +1$, where $\lfloor \cdot \rfloor$ is the usual integer floor function. See \cite{MNOS04} for the degenerate case $2/h \in \mathbb{N}$.

First, we characterize the stable plateaux. Recall from Remark \ref{rem:Omega Omega-bar} that in Section \ref{sec3.1}, we are able to fully characterize all stable plateaux in the system, not only the ones in the essential subset $\overline\Omega$.

From \eqref{eq:H def}, the set of ground states is characterized as $\mathcal{S} = \{\boxplus\}$. It turns out that every stable plateau in the system is a singleton, as we now start to classify them.

We start with the collection of stable plateaux, $\mathscr{P}^1 = \mathscr{P}^{\star,1}$ (cf. \eqref{eq:P1 def}).
We let $\eta \in \mathscr{P}^{\star,1}$ if and only if the following two conditions hold simultaneously:
\begin{itemize}
\item every $+1$ cluster $\mathcal{C}$ of $\eta$ satisfies either
	\begin{itemize}
	\item $\mathcal C$ is a rectangle with all side lengths at least $2$;
	\item $\mathcal C$ is a horizontal strip (wrapping around the vertical boundary) whose width is not equal to $L-1$;
	\item $\mathcal C$ is a vertical strip whose width is not equal to $K-1$;
	\item $\mathcal C = V$;
	\end{itemize}
\item every two $+1$ clusters of $\eta$ are at least graph distance $3$ apart from each other.
\end{itemize}

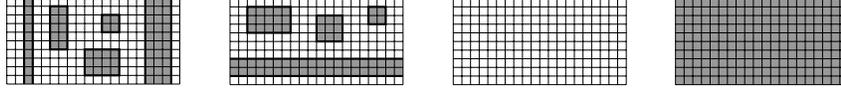
\begin{figure}[t]\centering
\begin{tikzpicture}[scale=0.115]
\fill[white] (-4,0) rectangle (16,10);

\fill[black!40!white] (-2,0) rectangle (-1,10);
\fill[black!40!white] (1,4) rectangle (3,9);
\fill[black!40!white] (5,1) rectangle (9,4);
\fill[black!40!white] (7,6) rectangle (9,8);
\fill[black!40!white] (12,0) rectangle (15,10);

\draw[very thin] (-4,0) grid (16,10);
\draw[thick] (-2,0)--(-2,10); \draw[thick] (-1,0)--(-1,10);
\draw[thick] (12,0)--(12,10); \draw[thick] (15,0)--(15,10);
\draw[thick] (1,4) rectangle (3,9);
\draw[thick] (5,1) rectangle (9,4);
\draw[thick] (7,6) rectangle (9,8);
\end{tikzpicture}
\hspace{5mm}
\begin{tikzpicture}[scale=0.115]
\fill[white] (-4,0) rectangle (16,10);

\fill[black!40!white] (-4,1) rectangle (16,3);
\fill[black!40!white] (-2,6) rectangle (3,9);
\fill[black!40!white] (6,5) rectangle (9,8);
\fill[black!40!white] (12,7) rectangle (14,9);

\draw[very thin] (-4,0) grid (16,10);
\draw[thick] (-4,1)--(16,1); \draw[thick] (-4,3)--(16,3);
\draw[thick] (-2,6) rectangle (3,9);
\draw[thick] (6,5) rectangle (9,8);
\draw[thick] (12,7) rectangle (14,9);
\end{tikzpicture}
\hspace{5mm}
\begin{tikzpicture}[scale=0.115]
\fill[white] (-4,0) rectangle (16,10);

\draw[very thin] (-4,0) grid (16,10);
\end{tikzpicture}
\hspace{5mm}
\begin{tikzpicture}[scale=0.115]
\fill[white] (-4,0) rectangle (16,10);

\fill[black!40!white] (-4,0) rectangle (16,10);

\draw[very thin] (-4,0) grid (16,10);
\end{tikzpicture}\caption{\label{fig2}Examples of configurations in $\mathscr{P}^1$ if $h>0$. We illustrate each
configuration in the dual lattice in the sense that the gray (resp.
white) faces indicate the $+1$ (resp. $-1$) spins.}
\end{figure}

See Figure \ref{fig2} for illustrations. Next, we declare nested subsets $\mathscr{P}^{\star,k}$ for $k \in [2,n_0+1]$, i.e., $\mathscr{P}^{\star,1} \supset \mathscr{P}^{\star,2} \supset \cdots \supset \mathscr{P}^{\star,n_0+1}$. Define $n_0 := \lfloor 2/h \rfloor$. From now on, by rectangle/strip we automatically refer to $+1$ rectangle/strip.
\begin{itemize}
\item For each $k \in [2,n_0]$, $\eta \in \mathscr{P}^{\star,k}$ if and only if $\eta$ has neither a rectangle whose smaller side length is less than or equal to $k$, nor a strip of width $1$.
\item Let $\mathscr{P}^{\star,n_0+1} := \{ \boxminus, \boxplus \}$.
\end{itemize}
Thus, $\mathfrak{m} = n_0+1$.
The corresponding depths (cf. \eqref{eq:Gamma-star1 def} and \eqref{eq:Gamma-starh def}) are given as
\begin{equation}\label{eq:TS}
\Gamma^{\star,k} := kh \quad \text{for} \quad k \in [1,n_0-1], \quad \Gamma^{\star,n_0} := 2-h,\quad \text{and} \quad \Gamma^{\star,n_0+1} := \Gamma_h,
\end{equation}
where $\Gamma_h := 4(n_0+1) - h(n_0(n_0+1)+1)$.

The intuition behind all this is as follows. The first time scale $e^{\Gamma^{\star,1}\beta}$ corresponds to the initial depth, which is $h$, on which a rectangle with its smaller side length $2$ may shrink by flipping its two edge $+1$ spins to $-1$ spins and become a smaller rectangle, or a strip with width $1$ may vanish. Then, for $k \in [2,n_0-1]$, on time scale $e^{\Gamma^{\star,k}\beta}$, a rectangle with its smaller side length $k+1$ may shrink by flipping its $k+1$ edge $+1$ spins, which has depth $kh$. Next, on time scale $e^{\Gamma^{\star,n_0}\beta}$, a rectangle with all side length bigger than $n_0$ may grow by adding $+1$ spins on one of its four sides, or a strip of width at least $2$ may grow by adding $+1$ spins on one of its two sides, both of which have depth $2-h$.
Finally, on the last time scale $e^{\Gamma^{\star,n_0+1}\beta}$, we observe the final metastable transition from $\boxminus$ to $\boxplus$ by overcoming/following a series of paths described above, thereby get absorbed at the unique stable state $\boxplus$.

The main convergence result can be summarized as follows.

\begin{theorem}[Metastable hierarchy, Case 1]\label{thm:case1}
For each $k\in[1,n_0+1]$, there exists a limiting
Markov chain $\{\mathfrak{X}^{\star,k}(t)\}_{t\ge0}$ in $\mathscr{P}^{\star,k}$
such that the $e^{\beta\Gamma^{\star,k}}$-accelerated Glauber dynamics
converges to $\{\mathfrak{X}^{\star,k}(t)\}_{t\ge0}$ in the sense
of Theorem \ref{thm:main}.
\end{theorem}

Each dynamics $\{\mathfrak{X}^{\star,k}(t)\}_{t\ge0}$, $k \in [1,n_0+1]$, is a one-step absorbing Markov chain in the sense that every element is either absorbing or gets absorbed by a single jump.
For the exact definition of the jump rates, we refer the readers to \cite[Section 3]{BL11}.

\subsection{\label{sec3.2}Case 2: Glauber dynamics, zero external field}

In Section \ref{sec3.2}, we assume the contrary that $h=0$ in \eqref{eq:H def}. The first and (maybe) the easiest notable difference would be that, in this case, we have full symmetry between the two spins $\pm1$, thus both $\boxminus$ and $\boxplus$ become the most stable configurations; in other words, $\mathcal{S} = \{ \boxminus , \boxplus \}$.
It is quite remarkable in that the situation here is far more complicated, since the energy landscape becomes vast and flat between the stable configurations $\boxminus$ and $\boxplus$. This complexity and degeneracy prevented the experts from analyzing this zero external field case for more than 20 years after the first results for the positive external field case. The results in Section \ref{sec3.2} are mainly due to \cite{KS25}.

Define $\Gamma_L := 2L+2$. In the zero external field case, the barrier $\overline\Phi$ (cf. \eqref{eq:Phi-bar def}) between $\boxminus$ and $\boxplus$ becomes (cf. \cite[Theorem 2.1]{NZ19})
\[
\overline\Phi = \Phi (\boxminus , \boxplus) = \mathbb{H} (\boxminus) + \Gamma_L,
\]
thus we have $\overline\Omega = \{ \eta \in \Omega : \Phi (\mathcal{S},\eta) \le \mathbb{H} (\boxminus) + \Gamma_L \}$.

It turns out that inside the essential region $\overline\Omega$, there are only two time scales; $e^{2\beta}$ and $e^{\Gamma_L\beta}$. For simplicity, in this subsection we assume that $K>L$. We briefly remark the $K=L$ case at the end of Section \ref{sec3.2}.

For each $k \in \mathbb{T}_K$ and $v \in [0,K]$, denote by $\xi_{k,v} \in \Omega$ the configuration such that $\xi_{k,v}(x_1,x_2)=+1$ if $x_1 \in \{k, \dots , k+v-1 \} \subseteq \mathbb{T}_K$ and $\xi_{k,v}(x_1,x_2)=-1$ if $x_1 \in \{ k+v , \dots , k-1 \} \in \mathbb{T}_K$. In words, $\xi_{k,v}$ is the vertical strip configuration of width $v$ that starts from the horizontal coordinate $k \in \mathbb{T}_K$ (see Figure \ref{fig3} left). Note that $\xi_{k,0}=\boxminus$ and $\xi_{k,K}=\boxplus$ for any $k \in \mathbb{T}_K$.

\begin{figure}[t]\centering
\begin{tikzpicture}[scale=0.115]
\fill[white] (-4,0) rectangle (16,10);

\fill[black!40!white] (-2,0) rectangle (9,10);

\draw[very thin] (-4,0) grid (16,10);
\draw[thick] (-2,0)--(-2,10); \draw[thick] (9,0)--(9,10);
\end{tikzpicture}
\hspace{5mm}
\begin{tikzpicture}[scale=0.115]
\fill[white] (-4,0) rectangle (16,10);

\fill[black!40!white] (3,0) rectangle (8,10);

\draw[very thin] (-4,0) grid (16,10);
\draw[thick] (3,0)--(3,10); \draw[thick] (8,0)--(8,10);
\end{tikzpicture}
\hspace{5mm}
\begin{tikzpicture}[scale=0.115]
\fill[white] (-4,0) rectangle (16,10);

\fill[black!40!white] (4,0) rectangle (6,4);
\fill[black!40!white] (4,5) rectangle (6,10);

\draw[very thin] (-4,0) grid (16,10);
\draw[thick] (4,0)--(4,4)--(6,4)--(6,0);
\draw[thick] (4,10)--(4,5)--(6,5)--(6,10);
\end{tikzpicture}
\hspace{5mm}
\begin{tikzpicture}[scale=0.115]
\fill[white] (-4,0) rectangle (16,10);

\fill[black!40!white] (4,0) rectangle (6,1);
\fill[black!40!white] (4,2) rectangle (5,4);
\fill[black!40!white] (5,5) rectangle (6,6);
\fill[black!40!white] (5,7) rectangle (6,9);
\fill[black!40!white] (4,9) rectangle (6,10);

\draw[very thin] (-4,0) grid (16,10);
\draw[thick] (4,0)--(4,1)--(6,1)--(6,0);
\draw[thick] (4,2) rectangle (5,4);
\draw[thick] (5,5) rectangle (6,6);
\draw[thick] (4,10)--(4,9)--(5,9)--(5,7)--(6,7)--(6,10);
\end{tikzpicture}\caption{\label{fig3}Configurations $\xi_{2,11}$ and $\xi_{7,5}$ which belong to $\mathscr{P}^1$ (left), and examples of atypical saddle configurations (right) if $h=0$.}
\end{figure}
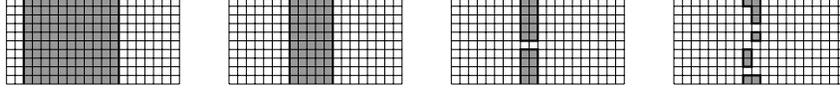

Then, the collection $\mathscr{P}^1$ of stable plateaux is characterized as follows.

\begin{theorem}[Structure of the stable plateaux, Case 2]
\begin{enumerate}
\item We have $\mathscr{P}^{\star,1} = \{ \boxminus , \boxplus \} \cup \{ \xi_{k,v} : k \in \mathbb{T}_K, \ v \in [2,K-2] \}$ and $\Gamma^{\star,1} = 2$.
\item At the next level, $\mathscr{P}^{\star,2} = \{ \boxminus , \boxplus \}$ and $\Gamma^{\star,2} = \Gamma_L = 2L+2$.
\end{enumerate}
\end{theorem}

We briefly explain the detailed metastable transitions. Starting from $\xi_{k,v}$ for some $k\in \mathbb{T}_K$ and $v\in [3,L-3]$, following a path of depth $2$, we may either attach a stick on the left or right, or detach a stick from the left or right, by adding/removing spins $+1$ consecutively one by one, eventually arriving at either $\xi_{k',2}$ or $\xi_{k',K-2}$ for some $k'$. From these \emph{edge} configurations, we arrive at $\boxminus$ or $\boxplus$ which are absorbing, but the path is much more complicated; this is because other than the typical one-dimensional mechanism described in the last sentence, there also exist atypical transition paths involving weird saddle configurations, as illustrated in Figure \ref{fig3} right.

Then, on the second time scale $e^{\Gamma_L \beta}$, we observe a simple random walk between $\boxminus$ and $\boxplus$ where the jump rate solves a certain variational formula. We refer to the detail in \cite[Theorem 3.4]{KS25} and summarize the result below.

\begin{theorem}[Metastable hierarchy, Case 2]
For each $\mathfrak{h} = 1,2$, there exists a limiting
Markov chain $\{\mathfrak{X}^{\star,\mathfrak{h}}(t)\}_{t\ge0}$ in $\mathscr{P}^{\star,\mathfrak{h}}$
such that the $e^{\beta\Gamma^{\star,\mathfrak{h}}}$-accelerated Glauber dynamics
converges to $\{\mathfrak{X}^{\star,\mathfrak{h}}(t)\}_{t\ge0}$ in the sense
of Theorem \ref{thm:main}.
\end{theorem}

\begin{remark}[Case $K=L$]
If $K=L$, then the overall structure gets doubled since the analogous horizontal strip configurations (which are $\xi_{k,v}$'s transposed in the diagonal direction) also attain the same energy level $\mathbb{H}(\boxminus)+\Gamma_L$. Nevertheless, the complexity of the structure remains the same since the vertical configurations and the horizontal configurations do not interact with each other inside the essential set $\overline\Omega$. In addition, the jump rate between $\boxminus$ and $\boxplus$ on the second time scale gets doubled since we have twice more options for the metastable transition trajectories. See \cite[Definition 6.9]{KS25} for the precise definition of this transposing procedure.
\end{remark}

\begin{remark}[Other models]
We conclude this subsection by noting that one can also consider generalized \emph{Potts}-type models, where both the positive/zero external field mechanisms may arise at the same time. The deepest metastable transitions are quite well understood at this point; see \cite{KS22} for the zero external field case, \cite{BGN22,BGN24} for partially non-degenerate case, \cite{Ahn24} for fully non-degenerate case, and \cite{BGK23} for general interaction constants. It would be interesting in these models to analyze all the remaining shallower-valley structures and thereby quantify the complete hierarchical structure of metastability.
\end{remark}

\subsection*{Kawasaki dynamics}

The \emph{Kawasaki dynamics}
in $\Omega$ is the continuous-time Markov chain $\{\eta_{\beta}^{\rm Ka}(t)\}_{t\ge0}$
whose transition rate function $r_{\beta}^{\rm Ka}:\Omega\times\Omega\to[0,\infty)$
is defined as
\begin{equation}
r_{\beta}^{\rm Ka}(\eta,\xi):=\begin{cases}
e^{-\beta\max\{\mathbb{H}(\eta^{x\leftrightarrow y})-\mathbb{H}(\eta),0\}} & \text{if}\quad\xi=\eta^{x\leftrightarrow y}\ne\eta\quad\text{and}\quad\{x,y\}\in E,\\
0 & \text{otherwise}.
\end{cases}\label{eq:r-beta-def}
\end{equation}
Here, $\eta^{x\leftrightarrow y}$ is obtained from $\eta$ by exchanging
the spins at $x$ and $y$:
\begin{equation}
\eta^{x\leftrightarrow y}(x)=\eta(y),\quad\eta^{x\leftrightarrow y}(y)=\eta(x),\quad\text{and}\quad\eta^{x\leftrightarrow y}(z)=\eta(z)\quad\text{for}\quad z\ne x,y.\label{eq:eta-xy def}
\end{equation}
According to the Kawasaki dynamics, each particle in the lattice jumps
independently to its vacant neighbor with rate $1$ if the jump does
not increase the energy and with exponentially small rate $e^{-\beta\Delta}$
if the jump increases the energy by $\Delta>0$.

Since the dynamics conserves the total number of each spin, we must restrict \emph{a priori} the number of $+1$ spins to obtain an ergodic system.
For a positive integer $\mathscr{N}$, the restricted configuration
space $\Omega_{\mathscr{N}}$ is defined as
\[
\Omega_{\mathscr{N}}:=\{\eta=(\eta(x))_{x\in V}\in\{-1,+1\}^{V}: |\{x \in V: \eta(x)=+1 \}| = \mathscr{N}\},
\]
such that the graph structure on $\Omega_{\mathscr{N}}$ now becomes connected.
In the particle system point of view, $\mathscr{N}$ indicates the total number of particles.

\subsection{\label{sec3.3}Case 3: Kawasaki dynamics, few particles}

In this subsection, mainly following \cite{BL15b}, assume that $\mathscr{N} < L^2/4$, i.e., the number of particles is bounded above.
Note that in the Hamiltonian formula \eqref{eq:H def}, the external field part does not play a role (since $\mathscr{N}$ is fixed) and we may use the following simplified form:
\[
\mathbb{H}(\eta) =- \frac12 \sum_{\{x,y\}\in E}\eta(x)\eta(y) .
\]
To illustrate the ideas, let $\mathscr{N} = n^2$ for some $n < L/2$.

For each $x \in V$, denote by $\eta^x$ the configuration that satisfies $\eta^x(y) = +1$ if and only if $y=x + (i,j)$ for some $i,j \in [0,n-1]$. Namely, $\eta^x$ is the configuration with a single $+1$ square with side length $n$ and starting from $x$. Then, define
\[
\Omega^0 = \{ \eta^x : x \in V \}.
\]
One can apply the usual isoperimetric inequalities to prove that $\Omega^0$ is the minimizer set of $\mathbb{H}$, i.e., $\mathcal{S} = \Omega_0$. Let $\mathbb{H}_0 := \mathbb{H}(\eta_x)$ for any $x \in V$.

Next, let $\Omega^1$ be the set of configurations with energy $\mathbb{H}_0+1$ that are first reachable by a path from $\Omega^0$ with energy barrier $\mathbb{H}_0+2$. A configuration in $\Omega^1$ takes form of a unique $n \times n$ rectangle, but one corner spin $+1$ gets removed and be attached at one of the edges.

Then, for $i \in [1,3]$, let $\Omega^{i+1}$ be the set of configurations with energy $\mathbb{H}_0+1$, first reachable by a path from $\Omega^i$, with energy barrier $\mathbb{H}_0+2$ and not going back to $\Omega^{i-1}$. According to this recursive definition, starting from $\Omega^4$, one cannot reach new configurations; the only option is to go back to $\Omega^3$. We illustrate candidates for each collection $\Omega^i$, $i \in [0,4]$, in Figure \ref{fig4}, and refer to \cite[Sections 4.3--4.6]{BL15b} for the precise definitions and properties.

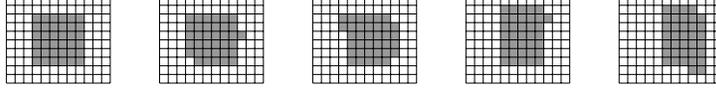
\begin{figure}[t]\centering
\begin{tikzpicture}[scale=0.115]
\fill[white] (0,0) rectangle (12,10);

\fill[black!40!white] (3,2) rectangle (9,8);

\draw[very thin] (0,0) grid (12,10);
\end{tikzpicture}
\hspace{5mm}
\begin{tikzpicture}[scale=0.115]
\fill[white] (0,0) rectangle (12,10);

\fill[black!40!white] (3,2) rectangle (9,8);
\fill[black!40!white] (9,5) rectangle (10,6);
\fill[white] (3,2) rectangle (4,3);

\draw[very thin] (0,0) grid (12,10);
\end{tikzpicture}
\hspace{5mm}
\begin{tikzpicture}[scale=0.115]
\fill[white] (0,0) rectangle (12,10);

\fill[black!40!white] (3,2) rectangle (9,8);
\fill[black!40!white] (9,3) rectangle (10,7);
\fill[white] (3,2) rectangle (4,6);

\draw[very thin] (0,0) grid (12,10);
\end{tikzpicture}
\hspace{5mm}
\begin{tikzpicture}[scale=0.115]
\fill[white] (0,0) rectangle (12,10);

\fill[black!40!white] (4,2) rectangle (9,9);
\fill[black!40!white] (9,7) rectangle (10,8);

\draw[very thin] (0,0) grid (12,10);
\end{tikzpicture}
\hspace{5mm}
\begin{tikzpicture}[scale=0.115]
\fill[white] (0,0) rectangle (12,10);

\fill[black!40!white] (5,2) rectangle (9,9);
\fill[black!40!white] (8,1) rectangle (10,8);

\draw[very thin] (0,0) grid (12,10);
\end{tikzpicture}\caption{\label{fig4}Configurations belonging to $\Omega^{0}$, $\Omega^{1}$, $\Omega^{2}$, $\Omega^{3}$, and $\Omega^{4}$, respectively.}
\end{figure}

The classification in the previous paragraph implies that (cf. \eqref{eq:Phi-bar def}) $\overline\Phi = \mathbb{H}_0 +2$, and that\footnote{Here, we disregarded set-theoretic issues and simply listed the configurations belonging to a stable plateau.}
\[
\mathscr{P}^{\star,1} = \Omega^0 \cup  \Omega^1  \cup  \Omega^2  \cup  \Omega^3  \cup  \Omega^4 .
\]
Every stable plateaux in $\Omega^1\cup\Omega^2\cup\Omega^3\cup\Omega^4$ has the initial depth $1$, thus gets absorbed to some $\eta^x$ for some $x\in V$ on the first time scale $e^\beta$. Thus, the second time scale, $e^{2\beta}$, regarding the deep transitions between the ground states in $\mathcal{S}$, naturally becomes the finally time scale. We summarize the results below.

\begin{theorem}[Structure of the stable plateaux, Case 3]
\begin{enumerate}
\item We have $\mathscr{P}^{\star,1} = \Omega^0 \cup \Omega^1 \cup \Omega^2 \cup \Omega^3 \cup \Omega^4$ and $\Gamma^{\star,1} = 1$.
\item At the second (and last) level, $\mathscr{P}^{\star,2} = \Omega^0 = \{ \eta^x : x\in V\}$ and $\Gamma^{\star,2} = 2$.
\end{enumerate}
\end{theorem}

\begin{theorem}[Metastable hierarchy, Case 3]
For each $\mathfrak{h} = 1,2$, there exists a limiting
Markov chain $\{\mathfrak{X}^{\star,\mathfrak{h}}(t)\}_{t\ge0}$ in $\mathscr{P}^{\star,\mathfrak{h}}$
such that the $e^{\beta\Gamma^{\star,\mathfrak{h}}}$-accelerated Glauber dynamics
converges to $\{\mathfrak{X}^{\star,\mathfrak{h}}(t)\}_{t\ge0}$ in the sense
of Theorem \ref{thm:main}.
\end{theorem}

\subsection{\label{sec3.4}Case 4: Kawasaki dynamics, $\mathscr{N}>L^2/4$}

Finally, we consider the case where there
exist macroscopic number of particles in the system, i.e., $\mathscr{N} > L^2/4$.
In this case, the fundamental characteristic in Section \ref{sec3.3}, that
the ground states are those with a $+1$ square, breaks down;
here, in the ground state, particles
tend to line up in one direction, forming a one-dimensional strip
droplet. Because of this fact, the geometry of the saddle structure
changes dramatically. In short, there exists a \emph{phase
transition} in the shape of the ground states at a threshold $\mathscr{N}^{*}:=\frac{L^{2}}{4}$.

From now on, we assume that $\mathscr{N}=L\mathscr{N}_{0}$ for an integer $\mathscr{N}_{0}$
such that
\begin{equation}
\frac{L}{4}<\mathscr{N}_{0}<\frac{K}{2},\quad\text{thus}\quad\frac{L^{2}}{4}<\mathscr{N}<\frac{KL}{2}.\label{eq:N0 assump}
\end{equation}
To fix ideas, we assume that $K>L$ and briefly discuss the $K=L$ case in Remark \ref{rem:K=00003DL}.

\begin{figure}[t]\centering
\begin{tikzpicture}[scale=0.115]
\fill[white] (-4,0) rectangle (16,10);

\fill[black!40!white] (0,0) rectangle (12,10);

\draw[very thin] (-4,0) grid (16,10);
\end{tikzpicture}
\hspace{5mm}
\begin{tikzpicture}[scale=0.115]
\fill[white] (-4,0) rectangle (16,10);

\fill[black!40!white] (1,0) rectangle (12,10);
\fill[black!40!white] (0,3) rectangle (1,7);
\fill[black!40!white] (12,2) rectangle (13,8);

\draw[very thin] (-4,0) grid (16,10);
\end{tikzpicture}
\hspace{5mm}
\begin{tikzpicture}[scale=0.115]
\fill[white] (-4,0) rectangle (16,10);

\fill[black!40!white] (1,0) rectangle (12,10);
\fill[black!40!white] (0,0) rectangle (1,4);
\fill[black!40!white] (12,0) rectangle (13,4);
\fill[black!40!white] (12,6) rectangle (13,8);

\draw[very thin] (-4,0) grid (16,10);
\end{tikzpicture}
\hspace{5mm}
\begin{tikzpicture}[scale=0.115]
\fill[white] (-4,0) rectangle (16,10);

\fill[black!40!white] (0,0) rectangle (10,10);
\fill[black!40!white] (10,0) rectangle (11,9);
\fill[black!40!white] (11,0) rectangle (12,7);
\fill[black!40!white] (12,0) rectangle (13,4);

\draw[very thin] (-4,0) grid (16,10);
\end{tikzpicture}\caption{\label{fig5}Configuration $\sigma^{k}$ and examples of other stable plateaux.}
\end{figure}
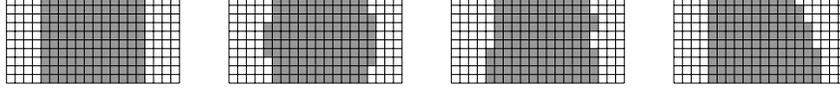

For each $k\in\mathbb{T}_{K}$, denote by $\mathfrak{c}^{k}$ the
$k$-th column in $\Lambda=\mathbb{T}_{K}\times\mathbb{T}_{L}$:
\begin{equation}
\mathfrak{c}^{k}:=\{k\}\times\mathbb{T}_{L}.\label{eq:ck def}
\end{equation}
Denote by $\sigma^{k}\in\Omega$ the configuration such
that (see Figure \ref{fig5})
\begin{equation}
\{x\in V:\sigma^{k}(x)=1\}=\mathfrak{c}^{k}\cup\mathfrak{c}^{k+1}\cup\cdots\cup\mathfrak{c}^{k+\mathscr{N}_{0}-1}.\label{eq:sigmak def}
\end{equation}
Then, the ground states are exactly the collection of $\sigma^k$'s, i.e., $\mathcal{S} = \{ \sigma^k : k \in \mathbb{T}_K \}$. As done in the previous subsection, write $\mathbb{H}_0 := \mathbb{H}(\sigma^k)$ for any $k \in \mathbb{T}_K$. Moreover, the energy barrier between the ground states is classified as
\[
\Phi(\sigma^{k},\sigma^{k'})=\mathbb{H}_{0}+4, \quad \text{thus} \quad \overline\Phi = \mathbb{H}_0 +4.
\]
Accordingly, $\overline\Omega = \{ \eta \in \Omega^{\mathscr{N}} : \Phi (\mathcal{S} ,\eta) \le \mathbb{H}_0+4 \}$.

It turns out that the remaining stable plateaux with higher energy have a highly complicated structure, even compared to all other three cases explained before. We decided to refer to \cite[Section 7]{Kim24} for the necessary but tedious long descriptions of these stable plateaux, and just present some examples in Figure \ref{fig5}. For the record, the main results are as follows.

\begin{theorem}[Characterization of the stable plateaux, Case 4]\label{thm:P1 char}
\begin{itemize}
\item \textbf{Energy $\mathbb{H}_{0}$}
\begin{itemize}
\item $\{\sigma^{k}\}$ for each $k\in\mathbb{T}_{K}$ (cf. \eqref{eq:sigmak def})
\end{itemize}
\item \textbf{Energy $\mathbb{H}_{0}+2$}
\begin{itemize}
\item $\{\sigma_{m;\ell,\ell'}^{k}\}$ for each $m\in[2,L-2]$ and
$\ell,\ell'\in\mathbb{T}_{L}$ (cf. \cite[Definition 7.4]{Kim24})
\item $\mathcal{S}_{1}^{k}$, $\mathcal{S}_{L-1}^{k-1}$ (cf. \cite[Definition 7.4]{Kim24}),
$\mathcal{R}^{k}$ and $\mathcal{L}^{k}$ (cf. \cite[Table 1]{Kim24})
\item $\{\eta\}$ for each $\eta\in\mathcal{R}_{(i)}^{k}\cup\mathcal{L}_{(i)}^{k}$
for $i\in[2,\frac{L}{2}]$ (cf. \cite[Table 3]{Kim24})
\end{itemize}
\item \textbf{Energy $\mathbb{H}_{0}+3$}
\begin{itemize}
\item $\{\eta\}$ for each $\eta\in\mathcal{D}_{m}^{k}$ for $m\in[2,L-2]$
(cf. \cite[Table 2]{Kim24})
\item stable plateaux in $\mathcal{R}_{m,\pm}^{k}\cup\mathcal{L}_{m,\mp}^{k}$
for $m\in[2,L-2]$ (cf. \cite[Table 2]{Kim24})
\item $\{\eta\}$ for each $\widehat{\mathcal{R}}_{(i)}^{k}\cup\widehat{\mathcal{L}}_{(i)}^{k}$
for $i\in[2,\frac{L-2}{2}]$ (cf. \cite[Table 3]{Kim24})
\item stable plateaux in each $\mathcal{R}_{(i),\pm}^{k}\cup\mathcal{L}_{(i),\mp}^{k}\cup\widehat{\mathcal{R}}_{(i),\pm}^{k}\cup\widehat{\mathcal{L}}_{(i),\mp}^{k}$
for $i\in[2,\frac{L+1}{2}]$\\ (cf. \cite[Table 3]{Kim24})
\end{itemize}
\end{itemize}
\end{theorem}

\begin{theorem}[Hierarchical decomposition of the stable plateaux, Case 4]
\label{thm:hier dec}$ $
\begin{enumerate}
\item We have $\Gamma^{\star,1}=1$, and $\mathscr{P}^{1}=\mathscr{P}^{\star,1}$
is decomposed as (cf. \eqref{eq:P-starh dec}
and \eqref{eq:P-starh C-starh rec})
\begin{align*}
\mathscr{P}_{{\rm rec}}^{\star,1}= & \bigcup_{k\in\mathbb{T}_{K}}\{\{\sigma^{k}\}\}\cup\bigcup_{k\in\mathbb{T}_{K}}\bigcup_{m=2}^{L-2}\{\{\sigma_{m;\ell,\ell'}^{k}\}:\ell,\ell'\in\mathbb{T}_{L}\}\\
 & \cup\bigcup_{k\in\mathbb{T}_{K}}\bigcup_{i\in[2,\frac{L}{2})}\{\{\eta\}:\eta\in\mathcal{R}_{(i)}^{k}\}\cup\bigcup_{k\in\mathbb{T}_{K}}\bigcup_{i\in[2,\frac{L}{2})}\{\{\eta\}:\eta\in\mathcal{L}_{(i)}^{k}\}\\
 & \cup\bigcup_{k\in\mathbb{T}_{K}}\bigcup_{\eta\in\mathcal{R}_{(\frac{L}{2})}^{k}}\{\{\eta\}\}\cup\bigcup_{k\in\mathbb{T}_{K}}\bigcup_{\eta\in\mathcal{L}_{(\frac{L}{2})}^{k}}\{\{\eta\}\}\quad\text{(if}\quad L\quad\text{is even)}.
\end{align*}
In particular, $\nu_{1}>1$.
\item We have $\Gamma^{\star,2}=2$, and $\mathscr{P}^{\star,2}$ is decomposed
as (cf. \eqref{eq:P-starh dec})
\[
\mathscr{P}_{{\rm rec}}^{\star,2}=\bigcup_{k\in\mathbb{T}_{K}}\{\{\sigma^{k}\}\}
\]
and
\begin{align*}
\mathscr{P}_{{\rm tr}}^{\star,2}= & \bigcup_{k\in\mathbb{T}_{K}}\bigcup_{m=2}^{L-2}\{\mathcal{S}_{m}^{k}\}\cup\bigcup_{k\in\mathbb{T}_{K}}\bigcup_{i\in[2,\frac{L}{2})}\{\mathcal{R}_{(i)}^{k}\}\cup\bigcup_{k\in\mathbb{T}_{K}}\bigcup_{i\in[2,\frac{L}{2})}\{\mathcal{L}_{(i)}^{k}\}\\
 & \cup\bigcup_{k\in\mathbb{T}_{K}}\bigcup_{\eta\in\mathcal{R}_{(\frac{L}{2})}^{k}}\{\{\eta\}\}\cup\bigcup_{k\in\mathbb{T}_{K}}\bigcup_{\eta\in\mathcal{L}_{(\frac{L}{2})}^{k}}\{\{\eta\}\}\quad\text{(if}\quad L\quad\text{is even)}.
\end{align*}
In particular, $\nu_{2}=k>1$.
\item We have $\Gamma^{\star,3}=4$, and $\mathscr{P}^{\star,3}$ is decomposed
as
\[
\mathscr{P}^{\star,3}=\mathscr{P}_{{\rm rec}}^{\star,3}=\{\{\sigma^{k}\}:k\in\mathbb{T}_{K}\}.
\]
In particular, $\nu_{3}=1$ thus $\mathfrak{m}=3$.
\end{enumerate}
\end{theorem}

\begin{theorem}[Metastable hierarchy, Case 4]
\label{thm:Kawasaki}For each $h\in[1,3]$, there exists a limiting
Markov chain $\{\mathfrak{X}^{\star,h}(t)\}_{t\ge0}$ in $\mathscr{P}^{\star,h}$
such that the $e^{\beta\Gamma^{\star,h}}$-accelerated Kawasaki dynamics
converges to $\{\mathfrak{X}^{\star,h}(t)\}_{t\ge0}$ in the sense
of Theorem \ref{thm:main}.
\end{theorem}

\begin{remark}[Case of $K=L$]
\label{rem:K=00003DL}Suppose here that $K=L$. Then, each stable
plateau classified in Theorem \ref{thm:P1 char} has a reflected counterpart
(by $\mathbb{T}_{K}\leftrightarrow\mathbb{T}_{L}$, which is possible
since $K=L$) with the same Hamiltonian value. Moreover, the hierarchical
metastable transitions in this reflected world remain the same as
characterized in Theorems \ref{thm:hier dec} and \ref{thm:Kawasaki},
but level $3$ would not be the terminal level since the original
ground states are not connected to the reflected ground states at
this level. Thus, the difference here occurs due to this additional
final transition, at level $\mathfrak{m}=4$, between the collection
of original ground states in $\mathcal{S}$ and the collection of
reflected new ground states. To see this, for simplicity we fix $k\in\mathbb{T}_{k}$,
recall $\sigma^{k}\in\mathcal{S}$, and define a new configuration
$\widehat{\sigma}^{k}$ defined as
\[
\{x\in V:\widehat{\sigma}^{k}(x)=1\}=\mathfrak{r}^{k}\cup\mathfrak{r}^{k+1}\cup\cdots\cup\mathfrak{r}^{k+\mathscr{N}_{0}-1},
\]
where $\mathfrak{r}^{\ell}$ is the $\ell$-th row in $\Lambda$.
Then, it holds that
\[
\Phi(\sigma^{k},\widehat{\sigma}^{k})\ge\mathbb{H}_{0}+5.
\]
It would be quite an interesting but technical question to characterize exactly this energy barrier $\Phi(\sigma^{k},\widehat{\sigma}^{k})$, and also to dive into the exact energy landscape lying between these transitions. We decide not to go any further in this direction and leave it as a future research topic.
\end{remark}

\begin{acknowledgement}
SK would like to thank all the organizers of the event
\emph{Particle Systems and PDE's XII}, especially Federico Sau (University of Trieste),
for the invitation and hospitality during his three weeks stay in the beautiful city of Trieste.
SK was supported by KIAS Individual Grant (HP095101) at the Korea
Institute for Advanced Study.
\end{acknowledgement}

\end{document}